\documentclass[a4paper,10pt]{amsart}
\usepackage[utf8x]{inputenc}
\usepackage{amsmath}
\usepackage{amsthm}
\usepackage{amssymb}
\usepackage{hyperref}
\usepackage{tikz}
\usepackage{amsfonts}
\usepackage[framemethod=TikZ]{mdframed}
\usetikzlibrary{positioning,arrows}
\usetikzlibrary{matrix}
\usetikzlibrary{shapes,snakes}
\usetikzlibrary{arrows,decorations.pathreplacing,patterns}

\mdfdefinestyle{Example}{%
    linecolor=blue,
    outerlinewidth=1pt,
    roundcorner=2pt,
    innertopmargin=\baselineskip,
    innerbottommargin=\baselineskip,
    innerrightmargin=20pt,
    innerleftmargin=20pt,
    backgroundcolor=gray!25!white}
    
\newtheorem{theorem}{Theorem}[section]

\newtheorem{definition}[theorem]{Definition}

\newtheorem*{theorem*}{Theorem}
\newtheorem*{corollary*}{Corollary}

\definecolor{darkblue}{rgb}{0.0,0,0.7} 
\newcommand{\darkblue}{\color{darkblue}} 
\newcommand{\defn}[1]{\emph{\darkblue #1}}

\def\C{\mathbb{C}}

\begin{document}

\title[A new cyclic sieving phenomenon for Catalan objects]{A new cyclic sieving phenomenon for Catalan objects}
\author{Marko Thiel}
\address{Department of Mathematics, University of Zürich, Winterthurerstrasse 190, 8050 Zürich, Switzerland}

\begin{abstract} Based on computational experiments, Jim Propp and Vic Reiner suspected that there might exist a sequence of combinatorial objects $X_n$, each carrying a natural action of the cyclic group $C_{n-1}$ of order $n-1$ such that the triple $\left(X_n,C_{n-1},\frac{1}{[n+1]_q}{2n \brack n}_q\right)$ exhibits the cyclic sieving phenomenon.
We prove their suspicion right.
\end{abstract}

\maketitle
\section{Introduction}
\subsection{The Cyclic Sieving Phenomenon}
Reiner, Stanton and White have observed that the following situation often occurs: one has a combinatorial object $X$, a cyclic group $C$ that acts on $X$ and a ``nice'' polynomial $X(q)$ whose evaluations at $|C|$-th roots of unity encode the cardinalities of the fixed point sets of the elements of $C$ acting on $X$. They termed this the cyclic sieving phenomenon.
\begin{definition}[\protect{\cite{reiner04cyclic}}]
 Let $X$ be a finite set carrying an action of a cyclic group $C$ and let $X(q)$ be a polynomial in $q$ with nonnegative integer coefficients.
 Fix an isomorphism $\omega$ from $C$ to the set of $|C|$-th roots of unity, that is an embedding $\omega:C\hookrightarrow\C^*$.
 We say that the triple $(X,C,X(q))$ exhibits the \defn{cyclic sieving phenomenon} (CSP) if
 \[|\{x\in X:c(x)=x\}|=X(q)_{q=\omega(c)}\text{ for every }c\in C.\]
\end{definition}
\noindent
In particular, if $(X,C,X(q))$ exhibits the CSP then $|X|=X(1)$. So $X(q)$ is a \defn{$q$-analogue} of $|X|$.
\subsection{Catalan numbers}
One of the most famous number sequences in combinatorics is the sequence $1,1,2,5,14,42,132,\ldots$ of \defn{Catalan numbers} given by the formula
\[C_n:=\frac{1}{n+1}\binom{2n}{n}.\]
A vast variety of combinatorial objects are counted by the Catalan number $C_n$, for example the set of triangulations of a convex $(n+2)$-gon and the set of noncrossing matchings of $\{1,2,\ldots,2n\}$.
The (MacMahon) \defn{$q$-Catalan number} $C_n(q)$ is the natural $q$-analogue of $C_n$, defined as
\[C_n(q):=\frac{1}{[n+1]_q} {2n \brack n}_q,\]
where $[n]_q:=1+q+q^2+\ldots+q^{n-1}$, $[n]_q!:=[1]_q[2]_q\cdots[n]_q$ and ${n \brack k}_q:=\frac{[n]_q!}{[n-k]_q![k]_q!}$.
It is a polynomial in $q$ with nonnegative integer coefficients.\\
\\
The $q$-Catalan number has the distinction of occurring in two entirely different CSPs for Catalan objects:
\begin{theorem}[\protect{\cite[Theorem 7.1]{reiner04cyclic}}]
 Let $\Delta_n$ be the set of triangulations of a convex $(n+2)$-gon and let $C_{\Delta_n}$ be the cyclic group of order $n+2$ acting on $\Delta_n$ by rotation.
 Then $(\Delta_n,C_{\Delta_n},C_n(q))$ exhibits the cyclic sieving phenomenon.
\end{theorem}
\begin{theorem}[\protect{\cite[Theorems 1.4 and 1.5]{petersen09promotion}}]\label{matching}
 Let $M_n$ be the set of noncrossing matchings of $[2n]:=\{1,2,\ldots,2n\}$ and let $C_{M_n}$ be the cyclic group of order $2n$ acting on $M_n$ by rotation.
 Then $(M_n,C_{M_n},C_n(q))$ exhibits the cyclic sieving phenomenon.
\end{theorem}
\noindent
Computational experiments by Jim Propp and Vic Reiner suggested that substituting an $(n-1)$-th root of unity into $C_n(q)$ always yields a positive integer.
So they suspected that there might also be cyclic sieving phenomenon involving $C_n(q)$ and a cyclic group of order $n-1$.
The main result of this note proves that their suspicion is correct.
\begin{theorem}\label{main}
 For any $n>0$, there exists an explicit set $X_n$ that carries an action of the cyclic group $C_{X_n}$ of order $n-1$ such that the triple $(X_n,C_{X_n},C_n(q))$ exhibits the cyclic sieving phenomenon.
\end{theorem}
\section{Proof of Theorem \ref{main}}
The first order of business is to define the set $X_n$. 
Call a subset of $[m]$ a \defn{ball} if it has cardinality $1$ and an \defn{arc} if it has cardinality $2$.
Define a \defn{$(1,2)$-configuration} on $[m]$ as a set of pairwise disjoint balls and arcs.
Say that a $(1,2)$-configuration $F$ has a \defn{crossing} if it contains arcs $\{i_1,i_2\}$ and $\{j_1,j_2\}$ with $i_1<j_1<i_2<j_2$.
If $F$ has no crossing it is called \defn{noncrossing}.
\begin{figure}[h]
\begin{center}
 \resizebox*{7cm}{!}{\includegraphics{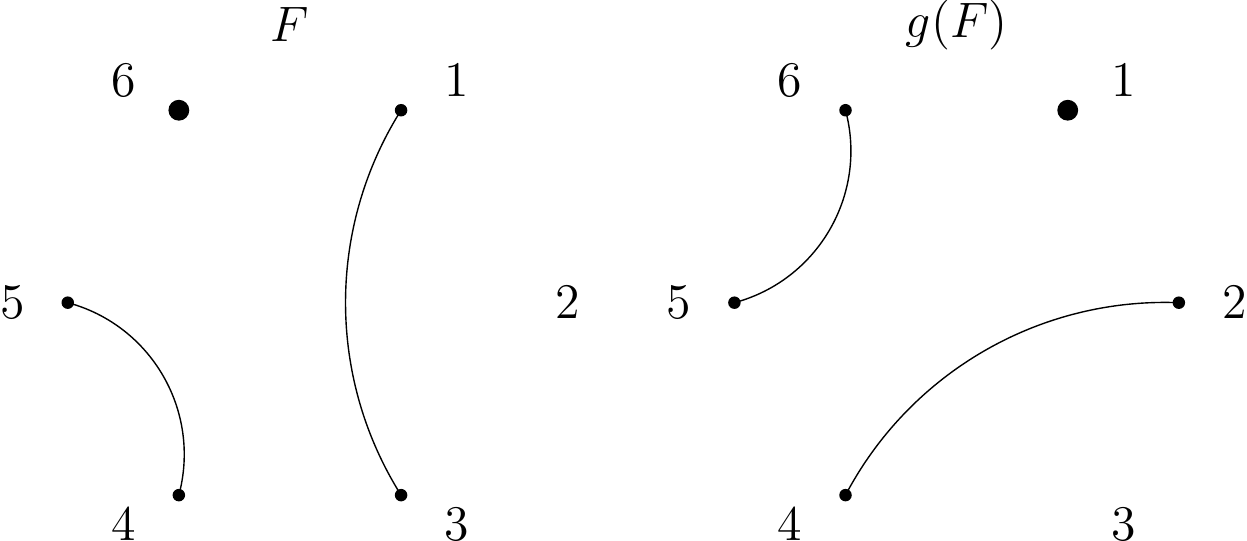}}\\
\end{center}
\caption{The noncrossing $(1,2)$-configuration $F=\{ \{1,3\},\{4,5\},\{6\} \}$ of $[6]$ and its rotation $g(F)$.}
\end{figure}\\
For $n>0$, define $X_n$ to be the set of noncrossing $(1,2)$-configurations of $[n-1]$. This is a corrected variant of $(\mathrm{e}^8)$ in Stanley's Catalan addendum \cite{stanleyaddendum}.
\begin{theorem}\label{cn}
 $|X_n|=C_n$ for all $n>0$.
 \begin{proof}
  To choose a noncrossing $(1,2)$-configuration $F$ of $[n-1]$, first pick the number $a$ of arcs in it.
  Then pick the subset $A$ of $[n-1]$ to be covered by arcs in one of $\binom{n-1}{2a}$ ways.
  Then choose a noncrossing matching of $A$ in one of $C_a=\frac{1}{a+1}\binom{2a}{a}$ ways.
  Finally choose the set of balls in $F$ from $[n-1]\backslash A$ in one of $2^{n-1-2a}$ ways.
  Thus
  \[|X_n|=\sum_{a\geq0}\binom{n-1}{2a}\frac{1}{a+1}\binom{2a}{a}2^{n-1-2a}=\frac{1}{n+1}\binom{2n}{n}.\]
  The last equality can be proven in many ways, for example using ``snake oil'' \cite{wilf06generating}.
 \end{proof}
\end{theorem}
Define $C_{X_n}$ as the cyclic group of order $n-1$ acting on $[n-1]$ by cyclically permuting its elements.
The corresponding action of $C_{X_n}$ on the set of $(1,2)$-configurations on $[n-1]$ preserves crossings, so it restricts to an action on $X_n$.
\begin{proof}[Proof of Theorem \ref{main}]
 We proceed by direct computation. Let 
 \begin{align*}
  g:[n-1]&\rightarrow[n-1]\\
  i&\mapsto i+1\text{ if }i\neq n-1\\
  n-1&\mapsto 1
 \end{align*}
 be a generator of $C_{X_n}$ and let $\omega:g^k\mapsto e^{\frac{2\pi ik}{n-1}}$ be an embedding $C_{X_n}\hookrightarrow\C^*$.
 In order to show that 
 \begin{equation}\label{CSP}
  |\{x\in X_n:g^k(x)=x\}|=C_n(q)_{q=e^{\frac{2\pi ik}{n-1}}}\text{ for every }k
 \end{equation}
 we simply compute both sides. Without loss of generality, we may assume that $k$ divides $n-1$, say $dk=n-1$.\\
 \\
 First we compute the right-hand side of (\ref{CSP}).
 If $d=1$, it equals $C_n(1)=C_n$. If $d=2$, it equals $\binom{n}{\frac{n-1}{2}}$ using $C_n(q)=\frac{1}{[n]_q}{2n \brack n+1}_q$ and \cite[Proposition 4.2 (iii)]{reiner04cyclic}.
 If $d\neq 1,2$ it equals $\binom{2k}{k}$ using $C_n(q)=\frac{[2n]_q}{[n]_q[n+1]_q}{2n-1 \brack n}_q$ and \cite[Proposition 4.2 (iii)]{reiner04cyclic}.\\
 \\
 Next we compute the left-hand side of (\ref{CSP}).
 To choose a noncrossing $(1,2)$-configuration $F$ of $[n-1]$ that fixed by $g^k$, first pick the number $a$ of points in $[k]$ that are covered by arcs of $F$.
 Then pick the subset of $[k]$ covered by arcs of $F$ in one of $\binom{k}{a}$ ways.
 The $g^k$-invariance of $F$ then determines the entire subset $A$ of $[n-1]$ covered by arcs of $F$. In particular $|A|=da$.
 Next choose a $g^k$-invariant noncrossing matching of $A$. These are in natural bijection with the $c^{a}$-invariant noncrossing matchings of $[da]$ (where $c$ is the generator of the natural cyclic action on $[da]$).
 So using Theorem \ref{matching} their number is $C_{\frac{da}{2}}(q)_{q=e^{\frac{2\pi ia}{da}}}$ (taken to be $0$ if $da$ is odd).
 Finally, choose the balls of $F$ in $[k]$ in one of $2^{k-a}$ ways. By $g^k$-invariance these determine all the balls of $F$.
 Putting it all together we have
 \begin{equation}
  |\{x\in X_n:g^k(x)=x\}|=\sum_{a\geq0}\binom{k}{a}C_{\frac{da}{2}}(q)_{q=e^{\frac{2\pi ia}{da}}}2^{k-a}
 \end{equation}
 If $d=1$, then 
 \[|\{x\in X_n:g^k(x)=x\}|=\sum_{a\geq0}\binom{n-1}{2a}\frac{1}{a+1}\binom{2a}{a}2^{n-1-2a}=\frac{1}{n+1}\binom{2n}{n}\]
 as in Theorem \ref{cn}.\\
 \\
 Now consider the case $d>1$. If $2\mid a$, then 
 \[C_{\frac{da}{2}}(q)_{q=e^{\frac{2\pi ia}{da}}}=\binom{a}{\frac{a}{2}}\]
 using \cite[Proposition 4.2 (ii)]{reiner04cyclic}. If $2\nmid a$, then using ${2n\brack n}_q-{2n\brack n+1}_q=q^nC_n(q)$ and \cite[Proposition 4.2 (ii)]{reiner04cyclic} gives
 \[C_{\frac{da}{2}}(q)_{q=e^{\frac{2\pi ia}{da}}}=\binom{a}{\frac{a-1}{2}}\text{ if }d=2,\]
 and 
 \[C_{\frac{da}{2}}(q)_{q=e^{\frac{2\pi ia}{da}}}=0\text{ if }d>2.\]
 So we calculate that for $d=2$ we have 
 \begin{align*}
  &|\{x\in X_n:g^k(x)=x\}|\\
  &=\sum_{a\geq0}\binom{\frac{n-1}{2}}{2a}\binom{2a}{a}2^{\frac{n-1}{2}-2a}
  +\sum_{a\geq0}\binom{\frac{n-1}{2}}{2a+1}\binom{2a+1}{a}2^{\frac{n-1}{2}-2a-1}\\
  &=\binom{n}{\frac{n-1}{2}}.
 \end{align*}
 For $d>2$ we have 
 \[|\{x\in X_n:g^k(x)=x\}|=\sum_{a\geq0}\binom{k}{2a}\binom{2a}{a}2^{k-2a}=\binom{2k}{k}\]
 as required.
\end{proof}
\section{Acknowledgements}
The author wishes to thank Jim Propp and Vic Reiner for making him aware of their suspicions via the Dynamical Algebraic Combinatorics mailing list.

\bibliographystyle{alpha}
\bibliography{literature}

\end{document}